\documentclass[12pt]{amsart}

\usepackage{tasks}
\usepackage{xcolor}
\usepackage{hyperref}
\usepackage{cleveref}
\usepackage{blindtext}
\usepackage{setspace}
\usepackage{amsmath,amssymb}
\usepackage{dirtytalk}
\usepackage[utf8]{inputenc}

\usepackage{graphicx}
\usepackage{subcaption}
\usepackage[utf8]{inputenc}
\usepackage[export]{adjustbox}
\usepackage{wrapfig}

\usepackage{marginnote}
\usepackage{datetime}
\usepackage{color}
\usepackage{amsrefs}
\usepackage{amsmath}
\usepackage{mathtools}
\usepackage{graphicx}  
\usepackage[top=3.5cm, bottom=2.5cm, outer=2.5cm, inner=2.5cm]{geometry}% margines
\usepackage{hyperref}
\hypersetup{linkcolor=red,
citecolor=black
}
\usepackage{amsmath}
\usepackage{hyperref}
\hypersetup{
	colorlinks = true,
	linkcolor = {magenta},
	citecolor = {blue}
}

\newtheorem{thm}{Theorem}[section]
\newtheorem{lem}[thm]{Lemma}
\newtheorem{prop}[thm]{Proposition}

\theoremstyle{definition}
\newtheorem{df}[thm]{Definition}

\newtheorem{AS}[thm]{Assumption}

\theoremstyle{remark}
\newtheorem{re}[thm]{Remark}

\numberwithin{equation}{section}

\usepackage[utf8]{inputenc}
\usepackage[english]{babel}
\usepackage{paralist}
\usepackage{amsthm}
 
\theoremstyle{definition}
\newtheorem{definition}{Definition}[section]
 
\theoremstyle{remark}

\newcommand{\Rey}{\text{Re}}
 \usepackage{bm,amsmath}

\newcommand{\RN}[1]{%
  \textup{\uppercase\expandafter{\romannumeral#1}}%
}

\newcommand{\bff}{\mathbf{f}}
\newcommand{\bfx}{\mathbf{x}}
\newcommand{\bfu}{\mathbf{u}}

\newcommand{\bfv}{\mathbf{v}}
\newcommand{\bfw}{\mathbf{w}}

\newcommand{\eps}{\langle \varepsilon \rangle}

\newcommand{\bfV}{\mathbf{V}}

\newenvironment{Alirev}{\color{blue}}{\color{black}}
\newcommand{\AAA}{\begin{Alirev}}
\newcommand{\PPP}{\end{Alirev}}

\allowdisplaybreaks[4]
\usepackage[utf8]{inputenc}

 \date{\today}
\begin{document}

%\today

\title{Statistics  in a Backscatter Eddy Viscosity Turbulence Model}

%    Information for first author

\author{Ali Pakzad$^*$}
\author{Farjana Siddiqua}
\thanks{$^*$Corresponding author. Email: \texttt{pakzad@csun.edu}}

%\author{Ali Pakzad\thanks{Corresponding author}
%\author{Farjana Siddiqua}

\maketitle
\setcounter{tocdepth}{1}

\begin{abstract}
This paper addresses two significant drawbacks of an eddy viscosity turbulence model: the issue of excessive dissipation relative to energy input and the lack of a universal parameter specification. Considering the Baldwin-Lomax model with backscatter effects, we first prove the existence and uniqueness of global weak solutions under mild conditions. Our next result shows that this model maintains energy dissipation rates consistent with energy input, thereby avoiding over-dissipation and aligning with $K41$ phenomenology. Additionally, we propose a range for the  model's mixing length.
\end{abstract}

\section{Introduction}

Eddy viscosity turbulence models, which intend to model averages instead of instantaneous of turbulent variables, are widely used for industrial, aeronautical, meteorological, and oceanographical applications \cites{P00,  W06, BIL06}. As all classical eddy viscosity models \cites{LN92, S07, S68, JL16}, the Baldwin-Lomax model \cite{BLM78} also has well-recognized limitations in not modeling backscatter or complex turbulence not at statistical equilibrium. Rong, Layton, and Zhao in \cite{RLZ19}  adopted the model to incorporate the effects of energy flow from fluctuations back to means (backscatter) for a fluid filling a domain with solid walls. With $\Omega \subset \mathbb{R}^3$ be a open domain bounded by $\partial \Omega$, the model is given by
\begin{equation}\label{BLM}
\begin{split}
\partial_t\bfu+   \beta^2   \,  \nabla \times (l^2(\bfx) \nabla \times \partial_t\bfu)  +  \nabla \cdot (\bfu \otimes \bfu) - &  \nu \Delta \bfu +  \nabla \times (l^2(\bfx) |\nabla \times \bfu| \nabla \times \bfu )+  \nabla q = \bff,  \\
& \nabla \cdot \bfu = 0,\ \bfx\in\Omega, \ t>0
\end{split}
\end{equation}
with an initial condition $ \bfu_0(\bfx)$, where $\beta$ is a positive model calibration parameter, and $l(\bfx)$ is a mixing length that depends on the distance to the wall, $0 \leq l(\bfx)  \rightarrow 0$ as $ \bfx \rightarrow \partial \Omega$. In the above $\bfu(\bfx,t)$ and $q(\bfx,t)$ approximate the true averages of Navier-Stokes velocity and pressure, respectively. The kinematic viscosity is denoted by $\nu$, and $\bff$ is the known body force. It is shown in \cite{RLZ19} that the eddy viscosity term $\nabla \times (l^2(\bfx) |\nabla \times \bfu| \nabla \times \bfu )$  accounts for the dissipative effect of the Reynolds stress. In addition, there is a two-dimensional computational test that indicates the non-smoothing dispersive term,  $\beta^2   \,  \nabla \times (l^2(\bfx) \nabla \times \partial_t\bfu)$,   accounts for statistical backscatter \footnote{{Statistical-backscatter refers to the movement of energy from fluctuations back to means when using ensemble averaging.}}, however there is no artificial negative viscosity.     Very recently, Berselli~\cite{MR4841282} also introduced and studied the rotational Smagorinsky model with backscatter effects.

%One main common drawback shared by many eddy viscosity models   is that model dissipation often does not balance with energy input and exceed it,  and  this may  lead to higher viscosity  (non-physical) solutions.    We prove that this is not so for the above model.  In this manuscript,   this concern is addressed  with a rigorous mathematical  analysis, supported by a three-dimensional computational evidence.

There are two common drawbacks shared by many eddy viscosity models. One is that model dissipation often exceeds the energy input, leading to non-physical solutions with higher viscosity. Another issue is the lack of a simple, universal, and effective specification of parameters, such as $\beta$ and $l(\bfx)$ in \eqref{BLM}, as they vary from case to case. In this manuscript, these two concerns are addressed through rigorous mathematical analysis, supported by computational evidence.

Herein,  we first study the existence and uniqueness of the global weak solution to the equations \eqref{BLM}, see Theorems \ref{Existencethm} and \ref{Uniquenssthm}. We show that, under mild condition on $\l(\bfx)$, the existence and uniqueness of solutions hold for the model. Another main result of this paper, \Cref{MainThm}, is that for the backscatter Baldwin-Lomax model \eqref{BLM} equipped with the Dirichlet boundary condition,  the over dissipation does not happen: the model's energy dissipation rate is consistent with its energy input rate, and that the model produces time averaged energy dissipation rates consistent with the $K41$ phenomenology given in \cite{K41}. In addition,  based on the analysis,  a range for the mixing length is proposed, $\l(\bfx) \sim \Rey^{-{1/2}}$, which is also consistent with the suggested mixing length in \cites{W06}.

\subsection{Related works} Studies of turbulence typically rely on statistical rather than pointwise descriptions, as providing a detailed account of fluid behavior in a turbulent region is impractical due to the chaotic structures spanning a wide range of length scales. One such statistical quantity is the time-averaged energy dissipation rate \cite{H72}. Based on Kolmogorov’s conventional turbulence theory at large Reynolds number, the energy dissipation rate per unit volume  $\langle \varepsilon \rangle $ should be independent of the kinematic viscosity $\nu$, see \cites{K41}. In other words, dissipation appears to exist independently of viscosity in the zero limit. By a dimensional consideration, the energy dissipation rate per unit volume scales as\footnote{ Based on the concept of energy cascade introduced by Richardson \cite{R22} in 1922,  the molecular viscosity is effective in dissipating the large scales' kinetic energy only at small scales.   Therefore,  the rate of dissipation is determined by the rate of transfer of energy from the largest eddies. The large eddies have energy of order $U^2$,  and the \say{rate} has dimensions $1/\text{time}$. The turnover time for the large eddies, i.e., the time that takes a large eddy with velocity $U$ to travel a distance $L$, is given by $\tau = L/U$. Thus the \say{rate of energy input} has dimensions $U^2/\tau = U^3/L.$}
$$\langle  \varepsilon\rangle  \simeq  C_{\varepsilon}\frac{U^3}{L},$$
where $U$ and  $L$ are  global velocity and length scales, with $C_{\varepsilon}=$  the asymptotic constant (Kolmogorov 1941). This result is fundamental to an understanding of turbulence \cite{S68}.

In the theory of turbulence, upper estimates of energy dissipation rates are useful for, in particular, overall complexities of turbulent flow simulations. It also determines the smallest persistent length scales and the dimension of any global attractor (if it exists) \cites{DFJ09, F95}. Doering and Constantin in \cite{DC92}   and Doering and Foias in \cite{DF02}  proved rigorous asymptotic upper bounds directly from the Navier-Stokes equations. Their bound is of the form  $$ \varepsilon \lesssim  \frac{U^3}{h},  \hspace{0.5cm} \text{as}  \hspace{0.1cm}\ \Rey \to \infty,  \hspace{0.5cm} \text{where} \hspace{0.2cm}  \Rey= \frac{UL}{\nu}.$$
These works have developed in many important directions,  e.g.,  \cites{AP17,AP19, AP20, W97, W00, W10, L02, L16, DR00,FS1,FS2} for the deterministic flows.   Very recently,  the effect of randomness on the dissipation rate is studied in \cite{fan20233d} and  \cite{FJP21}  when noise is added on a boundary of shear flow.   The authors in \cite{CP20} could also prove $\mathcal{O} \left( U^3 / L\right)$ lower bound for stochastically forced flows under an additional assumption of energy balance.

\subsection*{Organization of this paper} In \Cref{sec:S2}, we introduce some standard  notations and preliminaries. \Cref{sec:S3} revisits the corrected model \eqref{BLM}, and we present definitions and the setting for the analysis. Later in \Cref{sec:S4}, we discuss the existence and uniqueness results. Then in \Cref{sec:S5}, we prove the central results of these considerations, that the dissipation rate is independent of the viscosity at high Reynolds number. Numerical experiments, demonstrating and extending beyond the analytical results, are described in \Cref{sec:S6}. 

 \section{Notation \& preliminaries}\label{sec:S2}
 
 Let $\Omega \subset \mathbb{R}^3$ be a open domain bounded by $\partial \Omega$ with the volume $|\Omega|$ and $\bfx = (x_1, x_2, x_3) \in \mathbb{R}^3$. Constant $C$ represents a generic positive constant independent of $\nu, \ \Rey, $ and other model's parameters. Let $p \in [1 , \infty]$, and the Lebesgue space $\left(L^p(\Omega)\right)^3$ is the space of all measurable functions $\bfv$ on $\Omega$ for which
\begin{align*}
\|\bfv\|_{L^p} & :=\big( \int_{\Omega} |\bfv (\bfx)| ^p\, d\bfx\big)^{\frac{1}{p}} \, < \infty,  \hspace{1cm} \text{if} \hspace{0.2cm} p \in [1 , \infty),\\
\|\bfv\|_{L^{\infty}}& := \sup_{\bfx \in \Omega}  |\bfv (\bfx)| < \infty,  \hspace{2.3cm} \text{if} \hspace{0.2cm} p  = \infty.
    \end{align*}
The $ L^2$  norm and inner product will be denoted by $\|\cdot\|$ and $( \cdot ,  \cdot)$ respectively, while all other norms will be labeled with subscripts. Let $\bfV$ be a Banach space of functions defined on $\Omega$ with the associated norm $\| \cdot \|_\bfV$. We denote by $L^p(a, b; \bfV)$,  $p \in [1 , \infty]$,  the space of functions $\bfv : (a, b) \rightarrow  \bfV$ such that

\begin{align*}
\|\bfv\|_{L^p (a, b ; B)} &:=\big( \int_a^b \|\bfv(t)\|_B ^p\, dt\big)^{\frac{1}{p}} \, < \infty,  \hspace{1cm} \text{if} \hspace{0.2cm} p \in [1 , \infty),\\
\|\bfv\|_{L^{\infty} (a, b ; B)}& := \sup_{t  \in (a , b)} \|\bfv( t )\|_B < \infty,  \hspace{1.87cm} \text{if} \hspace{0.2cm} p  = \infty.
\end{align*} 
The space $H^{-1}$  is dual to $H^1$, and its norm is defined as

$$\|\bfv\|_{-1}= \sup_{\phi \in H^1}\frac{|(\bfv , \phi)|}{\|\nabla \phi\|}. $$
The space  $W_0^{1,p}(\Omega)$ consists of all functions in
$W^{1,p}$ which vanish on the boundary $\partial \Omega$ (in the sense of traces)
$$ W_0^{1,p} = \big\{ \bfv:   \hspace{0.1cm}\bfv \in  W^{1,p}(\Omega)  \hspace{0.3cm} \text{and} \hspace{0.3cm}\bfv |_{\partial \Omega}=0  \big\}.$$
Moreover
$$ W_{0, div}^{1,p} = \big\{ \bfv:   \hspace{0.1cm}\bfv \in  W_0^{1,p} (\Omega)  \hspace{0.3cm} \text{and} \hspace{0.3cm} \nabla \cdot \bfv=0 \big\}.$$

\subsection*{Some inequalities}
Let $1 \leq p \leq\infty$; we denote by $q$ the conjugate exponent,
$\frac{1}{p} + \frac{1}{q} = 1$. Assume that $f \in L^p$ and $g \in L^q$ with
$1\leq p \leq \infty$. Then $fg \in L^1$ and
\begin{equation}\label{Holder}
 \tag{H\"older inequality}
\|fg\|_{L^1} \leq \|f\|_{L^p}\, \|g\|_{L^q}
\end{equation} 
Moreover, for  any  $a, b \geq 0$ and $\lambda > 0$ we have
\begin{equation}\label{Young}
 \tag{Young inequality}
a b \leq \frac{\lambda^{-\frac{p}{q}}  \,a^p}{p} + \frac{\lambda \, b^q}{q}.
\end{equation}
There is a constant $C_{\text{P}}$ depends on $\Omega$ such that for any $\bfu \in \left(H_0^1 (\Omega)\right)^3,$
\begin{equation}\label{Poincare}
 \tag{Poincar\'e inequality}
 \| \bfu \| \leq C_{\text{P}} \| \nabla \bfu \|.
\end{equation}
In the sequel, we will make use of the classical embedding theorems for Sobolev spaces in $d=3$, 
\begin{equation}\label{Embedding}
 \tag{Sobolev embedding}
 \begin{split}
&W^{1,p}(\Omega)  \hookrightarrow L^6 (\Omega),  \quad \forall \, p \in [2, \infty).\\
 &W^{1,3}(\Omega) \hookrightarrow  L^p (\Omega), \quad \forall \, p \in [1, \infty).
 \end{split}
\end{equation}

Before presenting the main results, we report the following Gr\"onwall lemmas which were first mentioned in \cite{G19} by T. Hakon Gr\"onwall. 

\begin{prop}[\textbf{Gr\"onwall's lemma: Integral form}]\label{Gronwall}
Let $T \in \mathbb{R}^+ \cup \{\infty\}$ and $\alpha , \psi \in L^{\infty}(0, T)$ with  $\lambda \in L^1(0,T),  \lambda(t) \geq 0$,  for almost all $t\in [0, T]$. Then
$$\alpha(t) \leq \psi(t) + \int_0^t \lambda(s)\, \alpha(s)\, ds  \hspace{1cm} \text{a.e. in } [0,T],$$
implies for almost all $t \in [0,T]$

$$ \alpha(t) \leq \psi(t) + \int_0^t e^{\gamma(t) - \gamma(s)}\, \lambda(s)\, \psi(s)\, ds,$$
where $\gamma (t):= \int_0^t \lambda(\tau)\, d \tau$. If $\psi \in W^{1,1}(0,T)$, it follows

$$\alpha(t) \leq e^{\gamma(t)}\, \left(\psi(0)+ \int_0^t e^{-\gamma(s)}\, \psi'(s)\, ds\right).$$
Moreover, if $\psi$ is a monotonically increasing, continuous function, it holds,

$$\alpha(t) \leq e^{\gamma(t)}\, \psi(t).$$
\end{prop}

\begin{prop}[\textbf{Gronwall's lemma in differential form}]\label{GronwallDiff} Let $T \in \mathbb{R}^+ \cup \{\infty\}$,  $\alpha \in W^{1,1} (0, T)$ and $\psi,  \, \lambda  \in L^1 (0, T)$. Then

$$ \alpha'(t) \leq \psi(t) + \lambda (t)\, \alpha(t) \hspace{1cm} \text{a.e.  in } [0, T]$$
implies for almost all $t \in [0 , T]$

$$ \alpha(t) \leq \alpha(0)\, e^{\left[\int_0^t\, \lambda(\tau)\, d\tau\right] } + \int_0^t \, \psi(s) \, e^{\left[ \int_s^t\lambda(\tau)\, d \tau\right]}\, ds.$$
\end{prop}

Lastly, we would like to highlight an important property of the eddy viscosity term that will be valuable in the subsequent analysis. The proof of this property can be found in \cite{RLZ19}.

\begin{prop}[\textbf{Strong monotonicity}]\label{Monotonicity}
For all $\bfu, \bfu' \in W^{1,3}(\Omega)$, there is a constant $\tilde{c}>0$ such that 

$$\left( l^2(\bfx) |\nabla \times \bfu| \nabla \times \bfu - l^2(\bfx) |\nabla \times \bfu'| \nabla \times \bfu'\,  , \, \nabla \times (\bfu - \bfu')\right)  \geq  \tilde{c} \,  \|\l^{\frac{3}{2}}(\bfx)\, \nabla \times (\bfu - \bfu')\|_{L^3}^3. $$
\end{prop}

\section{The rotational backscatter model}\label{sec:S3}
One of the most used approaches in simulating turbulent flows is to model the ensemble-averaged Navier–Stokes equations by eddy viscosity,  and then solve the discretized result.  To begin, consider the incompressible Navier–Stokes equations (NSE) driven by a known body force $\bff (\bfx, t)$  and starting from an initial condition  $\bfv(\bfx , 0) = \bfv_0(\bfx) $
\begin{equation}\label{NSE}
\begin{split}
\partial_t \bfv+     \nabla \cdot (\bfv \otimes \bfv) - &  \nu \, \Delta \bfv +   \nabla p = \bff,  \hspace{0.2cm}\text{and} \hspace{0.2cm} \nabla \cdot \bfv = 0 \hspace{0.2cm}\text{in} \,  \Omega,\\
& \bfv = 0 \hspace{0.2cm} \text{on}  \hspace{0.15cm} \partial \Omega,  \hspace{0.45cm} \text{and}   \hspace{0.15cm} \int_{\Omega} p \, d\bfx  =0.
\end{split}
\end{equation}
Here, $ \Omega \subset \mathbb{R}^3$ is a bounded polyhedral domain,  $\bfv: \Omega \times [0; T] \rightarrow  \mathbb{R}^3$ is the  fluid velocity   and  $p : \Omega \times  (0; T] \rightarrow \mathbb{R} $ is the  fluid pressure. Due to measurement imprecision in input data (e.g., initial velocity), we consider an ensemble of $J$ Navier–Stokes equations. Let $\bfv(\bfx, t; \omega_j),  \, p(\bfx, t; \omega_j)$ be associated solutions to the NSE (\ref{NSE}) given an ensemble of initial conditions\footnote{The  sampled values $\omega_j$ can also incorporate variation in flow parameters not represented by the model such as variation of viscosity with temperature in an isothermal model.}
$$ \bfv(\bfx, 0; \omega_j)  = \bfv_0(\bfx; \omega_j), \hspace{1cm} j = 1, ..., J.$$
Then the instantaneous velocity field and pressure  are  decomposed into the mean and fluctuating components
$$\bfv  = \langle \bfv\rangle + \bfv',  \hspace{0.5cm} \text{and} \hspace{0.5cm} p  = \langle p \rangle + p',$$
where the ensemble means\footnote{The mean operator can be extended to any linear statistical filter that satisfies  the Reynolds rules  $$\langle \langle \phi \rangle \rangle =  \langle \phi \rangle \hspace{1cm} \text{and} \hspace{1cm} \langle \partial \phi \rangle = \partial \langle \phi \rangle $$ for any differential operator $\partial$.} are
$$
\langle \bfv\rangle \, (\bfx , t) := \frac{1}{J} \, \sum_{j=1}^J \bfv(\bfx, t; \omega_j),  \hspace{0.5cm} \text{and} \hspace{0.5cm}\langle p\rangle \, (\bfx , t)  := \frac{1}{J} \, \sum_{j=1}^J p(\bfx, t; \omega_j),$$
and fluctuations about the mean are given by
$$
\bfv'(\bfx, t; \omega_j) :=  \bfv(\bfx, t; \omega_j)  - \langle \bfv\rangle \, (\bfx , t), \hspace{0.5cm} \text{and} \hspace{0.5cm}  p'(\bfx, t; \omega_j) :=  p(\bfx, t; \omega_j)  - \langle p\rangle \, (\bfx , t).
$$

After introducing  the above decomposition into the equations \eqref{NSE}, one can obtain the ensemble averaging equations, which form a \textit{non-closed system}
\begin{equation}\label{Ens-NSE}
\partial_t \langle\bfv\rangle+    \langle\bfv\rangle \cdot \nabla \langle\bfv\rangle-  \nu \, \Delta \langle\bfv\rangle - \nabla \cdot \mathbf{R} \, (\bfv , \bfv) +  \nabla \langle p \rangle = \bff,  \hspace{0.2cm}\text{and} \hspace{0.2cm} \nabla \cdot \langle\bfu\rangle = 0 \hspace{0.2cm}\text{in} \,  \Omega,\\
\end{equation}
where the Reynolds stress $\mathbf{R} \, (\bfv, \bfv)$ is
$$ \mathbf{R} \, (\bfv, \bfv) := \langle\bfv\rangle \otimes \langle\bfv\rangle  -   \langle\bfv \otimes \bfv\rangle = -  \langle\bfv' \otimes \bfv'\rangle.$$
By invoking  the Boussinesq conjecture\footnote{Turbulent fluctuations have a dissipative effect on the mean flow \cite{B77}.} and   eddy viscosity hypothesis\footnote{This dissipativity aligns with the gradient  or the deformation tensor\cite{R57}.}, the Reynolds stress is modeled by 
$$- \nabla \cdot \mathbf{R} \, (\bfv, \bfv) \sim - \nabla \cdot \left( \nu_T \left(  \langle\bfv\rangle \right) \, \nabla \left(\langle\bfv\rangle \right) \right) + \text{terms incorporated into the new pressure},$$
and the standard eddy viscosity (EV) model is given by
\begin{equation}\label{EV}
\begin{split}
\partial_t \bfu+  \nabla \cdot (\bfu \otimes \bfu) - &  \nu  \, \Delta \bfu  - \nabla \cdot \big( \nu_T (  \bfu ) \, \nabla \bfu  \big) +  \nabla q = \bff,\\
& \nabla \cdot \bfu = 0.
\end{split}
\end{equation}
Since \eqref{EV} is a model, its solution is no longer the exact ensemble average, and solutions $\bfu(\bfx , t)$  and  $q(\bfx, t)$ are intended to approximate the true average velocity $\langle\bfv\rangle$ and pressure $\langle p \rangle$  receptively. 

The turbulent viscosity coefficient $\nu_T (  \langle\bfu\rangle )  $  should be calibrated by fitting it to flow data, and many parameterizations of eddy viscosity models are known, e.g.    \cites{BIL06, CL14, J06, LL02, MP94, V04}. In the simplified Baldwin-Lomax model \cite{BLM78}, the  eddy viscosity coefficient is calibrated as
\begin{equation} \label{BL-EV}
\nu_T (  \bfu ) =  l^2(\bfx) |\nabla \times \bfu|,
\end{equation} 
where $ l(\bfx)$ is a multiple of the distance from the boundary. The existence of weak solutions for the Baldwin-Lomax model in the steady case has recently been proved in \cite{BB20}.   Rotational LES models, in both steady and unsteady cases, have been recently developed and studied in~\cite{MR4841282, MR4568950}.  By looking into the energy equation, since   $\nu_T > 0$ in the above calibration (\ref{BL-EV}), the the eddy viscosity term  $ - \nabla \cdot \big( \nu_T (  \bfu ) \, \nabla \bfu  \big)$ in (\ref{EV}) can only characterize the dissipation effects of the Reynolds stress and is unable to represent the backscatter.  Like many EV models \cites{LN92, S07, S68, JL16},  the basic  Baldwin-Lomax model has difficulty with complex turbulence not at statistical equilibrium. It performs poorly in simulating backscatter and is unable to model effects such as separations and wakes. 

After expressing the eddy viscosity term in rotational (curl-curl) form, the authors in \cite{RLZ19} adopted the Baldwin-Lomax model to non-equilibrium turbulence, and introduced the backscatter rotational model \eqref{BLM}.  The derivation of the model is based on the variance evolution equation, and the approach suggested in \cite{JL16} (for more details see \cite{RLZ19} Section 3). It is proven that the effects of terms that modeled the Reynolds stress are dissipative in a long-time averaging sense
$$\liminf_{T\rightarrow \infty} \frac{1}{T} \, \int_0^T \, \int_\Omega \left( \beta^2  l^2(\bfx) \nabla \times \bfu_t \, \cdot \,  \nabla \times  \bfu    +  l^2(\bfx) | \nabla \times \bfu|^3 \right) \, d\bfx \, dt \geq 0,$$
while there are numerical evidences (see Figures 3, 4, 5, and 6 in \cite{RLZ19}) that indicate the backscatter occurs and energy is transferred from fluctuations to the mean at some times.

%Assuming that $l(\bfx) =  \mathcal{O}\left(\sqrt{d(\bfx, \partial \omega)}\right)$  near the boundary and $l(\bfx)$ be strictly positive inside the domain,  and if  $\bff \in L^2\left(0,  T ; L^2(\Omega)^3 \right)$ and $\bfu_0 \in W^{1,3}_{0, \delta} (\Omega)$, then the corrected Baldwin-Lomax model \eqref{BLM} equipped with the Dirichlet boundary conditions  has a unique regular weak solution \cite{BLN21}.

Existence and uniqueness of the weak solution are discussed in the following sections, and after formally taking the inner product of model \eqref{BLM} with $\bfu$, one can show that the  solution satisfies the energy equality 
\begin{equation}\label{Energy=}
\frac{1}{2} \frac{d}{dt}  \, \left( \,  \|\bfu\|^2 + \beta^2 \| \l(\bfx) \, \nabla \times \bfu\|^2\, \right) + \nu \| \nabla \bfu \|^2 + \int_{\Omega} l^2(\bfx) |\nabla \times \bfu|^3 \, d\bfx = (\bff , \bfu).
\end{equation}

\begin{re}
After dividing the both sides of \eqref{Energy=} by the volume of domain $|\Omega|$, we can identify the following quantities
\begin{itemize}
\item Model kinetic energy of mean flow per unit volume
$$\mbox{MKE}:=  \frac{1}{2} \frac{1}{|\Omega|}\,  \frac{d}{dt}  \, \left( \,  \|\bfu\|^2 + \beta^2 \| \l(\bfx) \, \nabla \times \bfu\|^2\, \right), $$
\item Rate of energy dissipation of mean flow per unit volume
$$ \varepsilon(t):= \frac{\nu}{|\Omega|}\,  \| \nabla \bfu \|^2 + \frac{1}{|\Omega|}\,  \int_{\Omega} l^2(\bfx) |\nabla \times \bfu|^3 \, d\bfx, $$
\item Rate of energy input to mean flow per unit volume
$$\frac{1}{|\Omega|}\, (\bff , \bfu).$$
\end{itemize}

\end{re}
  \begin{df}
 The dissipation quantities (per unit volume) are given  as
  $$\varepsilon_0 :=  \frac{1}{|\Omega|} \int_{\Omega} \nu  |\nabla  \bfu|^2 d \bfx =   \frac{1}{|\Omega|} \nu \|\nabla \bfu\|^2, $$
 
 $$\varepsilon_M := \frac{1}{|\Omega|}  \int_{\Omega} l^2(\bfx) |\nabla \times \bfu|^3 d\bfx.$$
 The long time average of a function $\phi (t)$ is defined
  $$\langle \phi \rangle :=  \limsup_{T \rightarrow \infty}  \frac{1}{T} \, \int_{0}^{T} \, \phi(t)\, dt. $$
The time-averaged energy dissipation rate (per volume) for model (\ref{BLM}) includes dissipation due to the viscous forces and turbulent diffusion. It is given by
\begin{equation}\label{EDR}
\langle \varepsilon \rangle :=  \limsup_{T \rightarrow \infty}  \frac{1}{T} \, \int_{0}^{T} \, \varepsilon_0 + \varepsilon_M\, dt. 
\end{equation}
  \end{df}

\section{On the existence and uniqueness  of weak solutions}\label{sec:S4}

In this Section, we prove the existence and uniqueness of (weak)  solutions to \eqref{BLM}. We assume that $\bfu$ is a sufficiently regular solution of the equations \eqref{BLM},  and then establish a priori estimates on $\bfu$ in the subsequent lemmas. In this context, our approach involves working in a formal manner. However, it is important to highlight that the formal derivation, along with the accompanying a priori estimates, can be rigorously established by employing a Galerkin approximation. Given the standard nature of this approach, which holds true for both the $2D$ and $3D$ Navier-Stokes (NS) system, as well as other turbulence models like the Smagorinsky \cite{S63} and Ladyzhynskaya model \cite{L63},  we have chosen to omit the intricate details at this juncture. Instead, we kindly refer readers seeking more comprehensive discussions to the following references: \cites{MNR01, T77}, as well as the supplementary references contained within.

\begin{AS}
Throughout the remainder of the manuscript, we assume the following regularity conditions on the external force $\bff$, initial condition $\bfu_0$, and the model parameter $\l(\bfx)$, 

\begin{equation}\label{Assumption}
\bfu_0 \in W_{0, div}^{1,3},  \hspace{0.5cm}  \bff \in L^2(0,T; L^2(\Omega)),   \hspace{0.5cm}  0 < \ell_0 \leq \l(\bfx) \leq \ell_{max}.
\end{equation}
\end{AS}

\begin{re}
  In the framework of rotational turbulence models, as in~\cite{MR4841282, MR4568950, MR4065455}, the mixing length \( \l(\mathbf{x}) \) is typically assumed to behave like the distance to the boundary; that is, \( 0 \leq \l(\mathbf{x}) \to 0 \) as \( \mathbf{x} \to \partial\Omega \) (see also~\cite{vandriest1956} for the behavior of turbulent flow near walls). This assumption is sufficient to prove the existence of solutions. However, to establish uniqueness, we impose a non-physical, artificial condition: \( 0 < \ell_0 \leq l(\mathbf{x}) \leq \ell_{\max} \), which enables us to obtain higher regularity. Removing this lower-bound restriction remains an interesting open problem.
\end{re}

Before proving the main theorems, we need to prove boundedness of the kinetic energy and that $\langle \varepsilon \rangle$ is well-defined in the following proposition. 

 \begin{prop}\label{prop1}
Under the above regularity assumptions \eqref{Assumption}, the kinetic energy and the time-averaged energy dissipation of the solution to (\ref{BLM})  are uniformly bounded in time. Moreover
 \begin{equation}\label{Bnd}
  \bfu \hspace{0.2cm} \text{and} \hspace{0.2cm} l (\bfx) \nabla \times \bfu \in L^\infty( 0 , \infty ; L^2).
   \end{equation}
 \end{prop}
 \begin{proof}
Using \ref{Poincare}  together with the Gr\"onwall inequality in (\ref{Energy=}) implies that the kinetic energy is uniformly
bounded in time according to
\begin{equation*}
\|\bfu (t)\|^2 + \beta^2 \| \l(\bfx) \, \nabla \times \bfu (t)\|^2 \leq \big( \|\bfu (0)\|^2 + \beta^2 \| \l(\bfx) \, \nabla \times \bfu (0)\|^2 \big) \, e^{-\alpha t}  + \frac{1}{\alpha \, \nu} \|\bff\|^2_{L^{2}(0 , t ; H^{-1})} \, (1 - e^{- \alpha t}), 
\end{equation*}
for any $ 0 \leq t < \infty$, where $ \alpha = \min \{\frac{\nu}{2 C_{\text{P}}} , \frac{\nu }{2 \beta ^2\, \ell_{\text{max}}}\} > 0$, for more details see \cite{RLZ19}. Hence, for a constant $C$ depending on the given data $\bfu_0$ and $\bff$,
 \begin{equation}\label{KEisBounded}
\begin{split}
&\sup_{t \in [0 , \infty]} \|\bfu(t)\|^2 \leq  C < \infty, \\
&\sup_{t \in [0 , \infty]}\| \l(\bfx) \, \nabla \times \bfu(t)\|^2 \leq  C < \infty.  
\end{split}
\end{equation}

Now \ref{Young} for $p= q = 2$, and $\lambda = \sqrt{ 2 \nu}$,  the right-hand side of (\ref{Energy=}) is majorized by
 \begin{equation*}
 (\bff , \bfu) \leq \|\bff \|_{H^{-1}} \,  \|\bfu \|_{H^{1}} \leq \frac{1}{4 \,\nu} \|\bff \|^2_{H^{-1}} + \nu  \|\nabla \bfu\|^2.
 \end{equation*}
 Therefore, 
  \begin{equation*}
 \frac{d}{dt}  \, \big(  \|\bfu\|^2 + \beta^2 \| \l(\bfx) \, \nabla \times \bfu\|^2\big) + \nu \| \nabla \bfu \|^2 + 2 \int_{\Omega} l^2(\bfx) |\nabla \times \bfu|^3 \, d\bfx  \leq  \frac{1}{4 \,\nu} \|\bff \|^2_{H^{-1}} .
\end{equation*}
Integrating above for $0$ to $T$, we obtain in particular
 \begin{equation}\label{L3Bound}
 \begin{split}
  \|\bfu(T)\|^2 + \beta^2 \| \l(\bfx) \, \nabla \times \bfu(T)\|^2 &  + \int_0^T  \nu \| \nabla \bfu \|^2  dt  + 2  \int_0^T\int_{\Omega} l^2(\bfx) |\nabla \times \bfu|^3 \, d\bfx \, dt  \\
&\leq   \|\bfu(0)\|^2 + \beta^2 \| \l(\bfx) \, \nabla \times \bfu(0)\|^2 + \frac{1}{4 \,\nu}  \, \|\bff \|^2_{L^2 (0 , T; H^{-1})}\\
& =: c_1(T). 
  \end{split}
\end{equation}
Now taking the limit superior from above, and using \eqref{Bnd}, one can also show that $\eps$ is well-defined.

 \end{proof}

\begin{lem}\label{L2Bound}
Under the regularity conditions specified in Assumption \ref{Assumption}, we have 

$$\bfu \in L^2(0,T; H^1(\Omega)) \cap L^3(0,T; W^{1,3}(\Omega)). $$

\begin{proof}
The result is concluded from \eqref{L3Bound}. 
\end{proof}

\end{lem}

\begin{lem}
Under the regularity conditions specified in Assumption \ref{Assumption}, we have
$$ \|\nabla \times \bfu(T) \|_{L^3}^3  \leq c_2(T).$$
\end{lem}
\begin{proof}
One starts by testing the momentum equation \eqref{BLM} with $\partial_t \bfu$.   As
\begin{align*}
   \left( \nabla \times (l^2(\bfx) |\nabla \times \bfu| \nabla \times \bfu \, ,\,  \partial_t \bfu )  \right) & =  - \int_{\Omega}  l(\bfx)^2 |\nabla \times \bfu| \,   \nabla \times \bfu \cdot \partial_t(\nabla \times \bfu) \, d \bfx\\
   &= -  \frac{1}{3}\,  \frac{d}{dt}\int_{\Omega}  l(\bfx)^2 (\nabla \times \bfu)^3 \, d \bfx, 
\end{align*}

this gives, after integration on $(0,T)$ and integration  by parts
\begin{equation}
\begin{split}
&\int_0^T \left( \|\partial_t\bfu\|^2 + \beta^2 \|\l(\bfx)\nabla \times \partial_t\bfu\|^2\right)\, dt + \frac{\nu}{2} \|\nabla \bfu(T)\|^2 + \frac{1}{3} \|\l(\bfx)^{\frac{2}{3}}\, \nabla \times \bfu(T) \|_{L^3}^3\\
&  =  \frac{\nu}{2} \|\nabla \bfu_0\|^2  +  \frac{1}{3} \|\l(\bfx)^{\frac{2}{3}}\, \nabla \times \bfu_0 \|_{L^3}^3 - \int_0^T (\bfu\cdot \nabla \bfu , \partial_t\bfu) \, dt + \int_0^T (\bff, \partial_t\bfu)\, dt. 
\end{split}
\end{equation}

Using the equivalency of the $L^p$ norms for $\nabla \bfu$ and $\nabla \times \bfu$, with  \ref{Holder} and  \ref{Embedding} gives

\begin{equation*}
\begin{split}
\int_{\Omega}  ( \bfu\cdot \nabla \bfu)^2\,   d\bfx  \leq \|\nabla \bfu\|_{L^3}^2 \, \| \bfu\|_{L^6}^2 \leq  c \,  \|\nabla \bfu\|_{L^3}^4 \leq  c \,  \|\nabla \times \bfu\|_{L^3}^4.
\end{split}
\end{equation*}

With the above bound, the convective term can be estimated  as

\begin{equation*}
\begin{split}
\int_0^T  \left(\bfu\cdot \nabla \bfu ,  \partial_t\bfu\right)  dt & \leq   \int_0^T \int_{\Omega}  \frac{1}{4}   (\partial_t\bfu)^2+ ( \bfu\cdot \nabla \bfu)^2\,   d\bfx\, dt \\
& \leq  \frac{1}{4} \int_0^T \|\partial_t\bfu\|^2 dt +  c  \int_0^T  \|\nabla \times \bfu\|_{L^3}^4 dt. 
\end{split}
\end{equation*}

The force term can be estimated as 

$$\int_0^T (\bff, \partial_t\bfu)\, dt \leq \int_0^T \|\bff\|^2 + \frac{1}{4}  \|\partial_t\bfu\|^2\, dt. $$

Then, by combining all the above estimates, we obtain
\begin{equation}
\begin{split}
&\int_0^T \left( \|\partial_t\bfu\|^2 + 2  \beta^2 \|\l(\bfx)\nabla \times \partial_t\bfu\|^2\right)\, dt + \nu \|\nabla \bfu(T)\|^2 + \frac{2}{3} \|\l(\bfx)^{\frac{2}{3}}\, \nabla \times \bfu(T) \|_{L^3}^3\\
&  \leq   \nu \|\nabla \bfu_0\|^2  +  \frac{2}{3} \|\l(\bfx)^{\frac{2}{3}}\, \nabla \times \bfu_0 \|_{L^3}^3 + 2 \int_0^T \|\bff\|^2 dt +  c  \int_0^T  \|\nabla \times \bfu\|_{L^3}^4 dt. 
\end{split}
\end{equation}
In particular with $ 0 < \ell_0\leq \l(\bfx)$ in Assumption \ref{Assumption}, we have
\begin{equation}
\begin{split}
  \|\nabla \times \bfu(T) \|_{L^3}^3 & \leq  \frac{3}{4} \frac{\nu}{\ell_0^2}  \|\nabla \bfu_0\|^2  +  \frac{1}{\ell_0^2} \|\l(\bfx)^{\frac{2}{3}}\, \nabla \times \bfu_0 \|_{L^3}^3 + \frac{3}{\ell_0^2} \int_0^T \|\bff\|^2 dt\\
&   +  \frac{c}{\ell_0^2}  \int_0^T  \|\nabla \times \bfu\|_{L^3}^4 dt. 
\end{split}
\end{equation}

The application of Gr\"onwall lemma (Proposition \ref{Gronwall}) gives

\begin{equation}
\begin{split}
  \|\nabla \times \bfu(T) \|_{L^3}^3 &  \leq \left(    \frac{3}{4} \frac{\nu}{\ell_0^2}  \|\nabla \bfu_0\|^2  +  \frac{1}{\ell_0^2} \|\l(\bfx)^{\frac{2}{3}}\, \nabla \times \bfu_0 \|_{L^3}^3 + \frac{3}{\ell_0^2} \int_0^T \|\bff\|^2 dt\right)\\
  & \times \exp \left(  \frac{c}{\ell_0^2}  \int_0^T  \|\nabla \times \bfu\|_{L^3} dt \right).  
\end{split}
\end{equation}
Using the result in \Cref{L2Bound}, we can estimate the term in the exponential as 

$$\int_0^T  \|\nabla \times \bfu\|_{L^3} \, dt  \leq  T^{\frac{2}{3}} \,  \left(\int_0^T  \|\nabla \times \bfu\|_{L^3}^3\, dt \right)^{\frac{1}{3}} \leq T^{\frac{2}{3}} \,  \left( c_1(T)\right)^{\frac{1}{3}},$$
which proves the lemma.
\end{proof}

Having the above estimates in hand, we employ the Galerkin approximation scheme to obtain the existence of weak solutions, as stated in the next theorem,  using standard arguments.  We refer the
reader to \cites{MNR01, T77}  for a detailed argument for the passage of the limit. 

\begin{thm}[\textbf{Existence of a weak solution}]\label{Existencethm}
Given the regularity conditions outlined in Assumption \ref{Assumption}, there exists a weak solution $\bfu$ to  \eqref{BLM}
satisfying

$$\bfu \in L^{\infty}(0,T; L^2(\Omega)) \cap L^2(0,T; H^1(\Omega)) \cap L^3(0,T; W_{0, div}^{1,3}(\Omega)), $$
 for all $T >0$. 
\end{thm}

\begin{thm}[\textbf{Uniqueness of the weak solution}]\label{Uniquenssthm}
Given the regularity conditions outlined in Assumption \ref{Assumption}, the weak solution of \eqref{BLM} is unique.
\end{thm}

\begin{proof}
 Let us assume that there are two weak solutions $\bfu$ and $\bfu'$ guaranteed by \Cref{Existencethm} with $\bfu_0 = \bfu'_0$, and denote $\bfw= \bfu - \bfu'$. Thus $\bfw \in W^{1,3}(\Omega)$ and $\bfw_0=0$, and the error function satisfies 
\begin{equation*}
\begin{split}
\frac{1}{2} \frac{d}{dt}  & \, \left( \,  \|\bfw\|^2 + \beta^2 \| \l(\bfx) \, \nabla \times \bfw\|^2\, \right) + \nu \| \nabla \bfw \|^2 + \left( \bfu \cdot \nabla \bfu - \bfu' \cdot \nabla \bfu' , \bfw \right)\\
& + \left( l^2(\bfx) |\nabla \times \bfu| \nabla \times \bfu - l^2(\bfx) |\nabla \times \bfu'| \nabla \times \bfu'\,  , \, \nabla \times (\bfu - \bfu')\right)  = 0. 
\end{split}
\end{equation*} 
After using the monotonicity property, \Cref{Monotonicity}, and also 
$$ (\bfu \cdot \nabla) \bfu   - (\bfu' \cdot \nabla) \bfu'  =  (\bfw \cdot \nabla) \bfu + (\bfu' \cdot \nabla) \bfw,$$
we obtain
\begin{equation*}
\begin{split}
\frac{1}{2} \frac{d}{dt}  & \, \left( \,  \|\bfw\|^2 + \beta^2 \| \l(\bfx) \, \nabla \times \bfw\|^2\, \right) + \nu \| \nabla \bfw \|^2 \leq - (\bfw \cdot \nabla \bfu , \bfw).
\end{split}
\end{equation*} 
By \ref{Holder} and \ref{Embedding}, the above nonlinear term can be estimated as 
\begin{equation*}
\begin{split}
(\bfw \cdot \nabla \bfu , \bfw)  \leq \|\nabla \bfu \|_{L^3} \| \bfw\|  \| \bfw\|_{L^6} & \leq  \|\nabla \bfu \|_{L^3} \| \bfw\|  \| \nabla \bfw\|\\
& \leq \frac{\nu}{2 } \| \nabla \bfw\|^2 + \frac{1}{2 \nu } \|\nabla \bfu \|_{L^3}^2 \| \bfw\|^2.  
\end{split}
\end{equation*}
Hence 
\begin{equation}\label{Eq1}
\begin{split}
 \frac{d}{dt}  & \, \left( \,  \|\bfw\|^2 + \beta^2 \| \l(\bfx) \, \nabla \times \bfw\|^2\, \right) + \nu \| \nabla \bfw \|^2 \leq  \frac{1}{ \nu } \|\nabla \bfu \|_{L^3}^2 \| \bfw\|^2.
\end{split}
\end{equation} 

With \Cref{GronwallDiff} in mind, denote  $\lambda(t)=   \nu^{-1} \|\nabla \bfu \|_{L^3}^2 $. In light of \Cref{L2Bound} and by using \ref{Holder}, one can show that $\lambda \in L^1(0,T)$  as 
$$ \frac{1}{ \nu } \int_0^T \|\nabla \bfu \|_{L^3}^2 \, dt  \leq   \frac{1}{ \nu } \, T^{\frac{1}{3}}\,  \|\bfu\|_{L^3(0,T; W^{1,3})}^2. $$
Finally, applying \Cref{GronwallDiff} to \eqref{Eq1} proves the uniqueness of the solution.

\end{proof}

\section{Statistics in the backscatter model}\label{sec:S5}
For the statistical study of the model, we restrict attention to time-independent applied forces in this paper, but the analysis could be extended to a wide variety of time-dependent forces.
 \begin{definition}\label{Def-Scales}  The scale of the body force $F$, large scale velocity  $U$ and length $L$, are defined, 
 $$F := \langle \frac{1}{|\Omega|} \| \bff\|^2 \rangle^{\frac{1}{2}}, $$
  $$U := \langle \frac{1}{|\Omega|} \| \bfu\|^2 \rangle^{\frac{1}{2}}, $$
  $$L := \min\Big\{ |\Omega|^{\frac{1}{3}} ,  \frac{F}{(\frac{1}{|\Omega|} \, \| \nabla \bff\|^2)^\frac{1}{2}} ,  \frac{F}{(\frac{1}{|\Omega|} \, \| \nabla \times \bff\|_{L^3}^3)^\frac{1}{3}}\Big\}.$$
 The dimensionless Reynolds number is given by
$$\Rey = \frac{U\, L}{\nu}.$$
 \end{definition}
 Next, we prove the central result of this paper, which is generally understood in terms of the cascade picture of turbulence: the energy dissipation rate tends to become independent of the viscosity in high-Reynolds-number turbulence.
\begin{thm}\label{MainThm}

Let $\bfu$  be a weak solution of the backscatter Baldwin-Lomax model \eqref{BLM} with homogeneous Dirichlet boundary conditions,  $\bfu = 0$ on $(0, T) \times \partial \Omega$,  satisfying the energy inequality \eqref{Energy=}.  Suppose the data   $\bff = \bff(\bfx)$ is smooth, divergence-free with zero mean functions, and $\bfu_0 \in L^2(\Omega)$. Then the time averaged energy dissipation rate per unit mass $\eps$ given by \eqref{EDR} satisfies 
\begin{equation}\label{FinalBd}
\eps \leq \Big( 1 + \frac{1}{\Rey} + \frac{1}{3}\, (\frac{\ell_{\text{max}}}{L})^2\Big)\, \frac{U^3}{L},
\end{equation}
 where  $\ell_{\text{max}} := \sup_{\bfx \in \Omega} |l (\bfx)|^2$.
\end{thm}
\begin{proof}
Calculation consists of two main steps. First, we find an upper bound on $\eps$ based on $F$ and $U$. In the second step,  $F$ is estimated, and the result is derived.  
\subsection*{Step 1.} Take  the time average of (\ref{Energy=}), noting that the time average of the time derivative vanishes, \Cref{prop1}, apply \ref{Holder} in time for $p = q = 2$, and divide by $|\Omega|$ to deduce
\begin{equation*}
\begin{split}
\frac{1}{T} \int_0^T \varepsilon \, dt & \leq  \mathsf{O}(\frac{1}{T}) + \frac{1}{\Omega} \, \frac{1}{T} \int_0^T (\bff , \bfu)\, dt  \\
& \leq   \mathsf{O}(\frac{1}{T}) + \big( \frac{1}{\Omega} \frac{1}{T} \int_0^T \|\bff\|^2 \, dt \big)^{\frac{1}{2}} \, \big( \frac{1}{\Omega} \frac{1}{T} \int_0^T \|\bfu\|^2 \, dt \big)^{\frac{1}{2}}.
\end{split}
\end{equation*}
Taking the limit superior, which exists by \Cref{prop1}, as $T \rightarrow \infty$ we obtain
\begin{equation}\label{EFU}
\eps \leq F \, U. 
\end{equation}

\subsection*{Step 2.} To bound the RHS,  we first find a bound on $F$ by taking the $L^2$ inner product of (\ref{BLM}) with $\bff$, integrating by parts, and averaging over $[0, T]$ and the domain $\Omega$.
\begin{equation}\label{F2}
\begin{split}
F^2 = \frac{1}{|\Omega|}\, \frac{1}{T} \, \int_0^T \big[ & (\bfu_t, \bff)+   \beta^2   \, ( l^2(\bfx) \nabla \times \bfu_t , \nabla \times \bff)  +    (\bfu \otimes \bfu , \nabla \bff )  \\
&  + \nu (\nabla \bfu , \nabla \bff) +  (l^2(\bfx) |\nabla \times \bfu| \nabla \times \bfu , \nabla \times \bff )  \big]\, \, dt.
\end{split}
\end{equation}

Next, we estimate each term on the RHS of the above equality as follows. The first  term $\rightarrow 0$ as $T \rightarrow \infty$ since 
\begin{equation}
\frac{1}{T} \, \int_0^T  (\bfu_t, \bff) \, dt  \simeq \mathsf{O} (\frac{1}{T}).
\end{equation}

The second term is also $\mathsf{O}(\frac{1}{T})$. This can be shown by using    H\"older's inequality  along with \eqref{Bnd} as follows
 
 \begin{equation}
 \begin{split}
 \left| \frac{1}{T} \,  \int_0^T ( l^2(\bfx) \nabla \times \bfu_t , \nabla \times \bff) \, dt \right| & =   \left| \frac{1}{T} \,  \int_0^T \partial_t( l^2(\bfx) \nabla \times \bfu , \nabla \times \bff) \, dt\right|\\
 &  \leq \ell_{\text{max}}\,  \|\nabla \times \bff \| \left[  \frac{\| l(\bfx) \, \nabla \times \bfu(T)\| + \| l(\bfx) \, \nabla \times \bfu_0\|}{T}\right].
 \end{split}
 \end{equation}

%\begin{equation*}
%\begin{split}
%\frac{1}{T} \,  \int_0^T ( l^2(\bfx) \nabla \times \bfu_t , \nabla \times \bff) \, dt     & \leq  \frac{1}{T} \, \int_0^T \| l^2(\bfx) \nabla \times \bfu_t\|\, \| \nabla \times \bff \| \, dt \\
%& \leq \big( \frac{1}{T} \,  \int_0^T  \| l^2(\bfx) \nabla \times \bfu_t\|^2 \, dt \big)^{\frac{1}{2}} \, \big( \frac{1}{T} \,  \int_0^T \| \nabla \times \bff \|^2 \, dt \big)^{\frac{1}{2}}
%\end{split}
 %\end{equation*}
  %While the force term in the above inequality is under control, we can show that the first term on the RHS  is of order $\frac{1}{T}$,
 
% \begin{equation*}
%\begin{split}
 %\frac{1}{T} \,  \int_0^T  \| l^2(\bfx) \nabla \times \bfu_t\|^2 \, dt  & =  \frac{1}{T} \,  \int_0^T   \int_{\Omega}| l^2(\bfx) \nabla \times \bfu_t|^2  d\bfx\, dt \\
% & \leq \ell_{\text{max}}^4 \, \frac{1}{T} \, \int_{\Omega} \int_0^T |\nabla \times \bfu_t|^2 \, dt \, d\bfx \\
% & \leq \ell_{\text{max}}^4   \, \int_{\Omega} \|\nabla \times \bfu_t\|_{L^{\infty}(0 , T)}  \, \big[   \frac{1}{T}\int_0^T  \frac{\partial}{\partial t}(\nabla \times \bfu)\, dt\big ] \, d\bfx  \\
 %&  =  \ell_{\text{max}}^4  \, \int_{\Omega} \|\nabla \times \bfu_t\|_{L^{\infty}(0 , T)}  \,  \big[  \nabla \times \frac{\bfu (T) - \bfu_0}{T}\big]\, d\bfx
%\end{split}
 %\end{equation*}
 
 In the long time limit $T \longrightarrow \infty$, the third and the fourth terms are bounded using the Cauchy-Schwarz-Young inequality and Definition (\ref{Def-Scales}). 
 \begin{equation}\label{Est3}
 \begin{split}
\left| \frac{1}{|\Omega|}\, \frac{1}{T} \, \int_0^T  (\bfu \otimes \bfu , \nabla \bff )    \, dt \right| & \leq \|\nabla \bff\|_{L^\infty} \, \frac{1}{|\Omega|}\, \frac{1}{T} \, \int_0^T  \|\bfu\|^2\, dt\\
 &\leq \frac{F}{L}\, U^2.
 \end{split}
 \end{equation}
 Moreover 
  \begin{equation}\label{Est4}
 \begin{split}
 \left|\frac{1}{|\Omega|}\, \frac{1}{T} \, \int_0^T \nu  (\nabla \bfu , \nabla \bff) \, dt  \right| & \leq \big(  \frac{1}{T} \, \frac{1}{|\Omega|} \, \int_0^T \nu \|\nabla \bfu\|^2 \, dt\big)^{\frac{1}{2}} \, \nu^{\frac{1}{2}} \,  \big(  \frac{1}{T} \, \frac{1}{|\Omega|} \, \int_0^T \| \nabla\bff\|^2 \, dt\big)^{\frac{1}{2}} \\
 & \leq \langle \varepsilon_0\rangle^{\frac{1}{2}} \,  \nu^{\frac{1}{2}} \,  \frac{F}{L} =  F \big[  \frac{\langle \varepsilon_0\rangle^{\frac{1}{2}}}{U^{\frac{1}{2}}} \, \, \frac{\nu^{\frac{1}{2}} \, U^{\frac{1}{2}}}{L}\big] \\
 & \leq F \big[ \frac{1}{2} \, \frac{\langle \varepsilon_0\rangle}{U} + \frac{1}{2} \frac{\nu \, U}{L^2}\big].
 \end{split}
 \end{equation}
 
The last term in \eqref{F2}  is estimated first  using H\"older’s inequality for $p=3, q = \frac{3}{2}$  and then considering  the scales  given in Definition  \ref{Def-Scales}  as 
\begin{equation*}
\begin{split}
\left| \frac{1}{|\Omega|}(l^2(\bfx) |\nabla \times \bfu| \nabla \times \bfu , \nabla \times \bff )\right| & \leq  \frac{1}{|\Omega|} \|\nabla \times \bff \|_{L^3} \, \big( \int_{\Omega} l^3(\bfx) |\nabla \times \bfu|^3 \, d\bfx \big)^{\frac{2}{3}} \\
& = \frac{1}{|\Omega|^{\frac{1}{3}}} \|\nabla \times \bff \|_{L^3} \, \big(  \frac{1}{|\Omega|} \int_{\Omega} l^3(\bfx) |\nabla \times \bfu|^3 \, d\bfx \big)^{\frac{2}{3}} \\
& \leq \frac{F}{L} \, \ell_{\text{max}}^{\frac{2}{3}} \, \big(  \frac{1}{|\Omega|} \int_{\Omega} l^2(\bfx) |\nabla \times \bfu|^3 \, d\bfx \big)^{\frac{2}{3}} \\
& = F \, \big[ \frac{\ell_{\text{max}}^{\frac{2}{3}} \, U^{\frac{2}{3}}}{L} \,  \, \, \left(\frac{\varepsilon_M}{U}\right)^{\frac{2}{3}} \big] \\
& \leq F \, \big[ \frac{1}{3} \frac{\ell_{\text{max}}^2 \, U^2}{L^3}  + \frac{2}{3} \frac{\varepsilon_M}{U} \big].
\end{split}
\end{equation*}
After taking long time averaged $\langle \cdot \rangle$ from above estimate, we obtain,
\begin{equation}\label{Est5}
\left| \frac{1}{|\Omega|} \, \frac{1}{T} \, \int_0^T\, (l^2(\bfx) |\nabla \times \bfu| \nabla \times \bfu , \nabla \times \bff )| \, dt \right| \leq  F  \, \big[ \frac{1}{3} \frac{\ell_{\text{max}}^2 \, U^2}{L^3}  + \frac{2}{3} \frac{\langle\varepsilon_M\rangle}{U} \big]. 
\end{equation}

Taking the limit superior $T \rightarrow \infty$ from (\ref{F2}), and using the estimates (\ref{Est3}), (\ref{Est4}) and (\ref{Est5}) gives,
\begin{equation}\label{F}
F \leq \frac{U^2}{L} + \frac{1}{2} \, \frac{\langle \varepsilon_0\rangle}{U} + \frac{1}{2} \frac{\nu \, U}{L^2} + \frac{1}{3} \frac{\ell_{\text{max}}^2 \, U^2}{L^3}  + \frac{2}{3} \frac{\langle\varepsilon_M\rangle}{U}.
\end{equation}

Now one can obtain the upper bound on $\eps$ by considering the  above estimate on (\ref{EFU}),  
$$
\eps \leq \Big( 1 + \frac{1}{\Rey} + \frac{1}{3}\, (\frac{\ell_{\text{max}}}{L})^2\Big)\, \frac{U^3}{L}.
$$
\end{proof}

\begin{re}
In the more physically relevant case of the vanishing viscosity limit, where \( \nu \to 0 \) but the eddy viscosity does not vanish, the above estimate can be rewritten as
\[
\eps \lesssim \left( 1 +  ( \frac{\ell_{\text{max}}}{L} )^2 \right)\, \frac{U^3}{L},
\]
which is consistent with Kolmogorov’s classical turbulence theory and also with the rate proven in \cite{DF02}.
\end{re}

\subsection{On the Model's parameter $\l(\bfx)$}
As in all other models, the model performance herein also depends on the value of the parameters chosen. The estimate \eqref{FinalBd} gives insight into the model's parameter  $\l(\bfx)$,  but not $\beta$. The eddy viscosity dissipation should be comparable to the pumping rate of energy to small scales by the nonlinearity, $U^3/L$, and to energy dissipation in the inertial range, $\Rey^{-1}\, U^3/L$. Hence, the following restriction is suggested 
\begin{equation}\label{MeshIndp}
\frac{L}{\Rey^{1/2}}  \lesssim \,  \l(\bfx) \, \lesssim L\, , \quad \text{mesh independent case}.
\end{equation}
The above range  $\l(\bfx) \sim \Rey^{- \frac{1}{2}}$ is consistent with the mixing length suggested in \cite{W06} (see Chapter 3),  and also the one used in  \cite{RLZ19} for the computational experiments.

To solve the equation, and after spatial discretization,  the smallest scale available is at the order of the mesh size $h$. On the other hand,  the size of the smallest persistent solution scales is given by Kolmogorov dissipation micro-scale $\Rey^{-{3/4}} \, L$. Hence, we can estimate $h \sim \Rey^{-{3/4}} \, L $, and the following estimate of mesh dependence case can be derived
\begin{equation}\label{MeshDep}
L^{\frac{1}{3}}\, h^{\frac{2}{3}}  \lesssim \,  \l(\bfx) \, \lesssim L\,, \quad \text{mesh dependent case}.
\end{equation}
Therefore based on the above bound, after discretization,  the mixing length can be chosen as $\l(\bfx) \sim h^{\frac{2}{3}}$, in contrast with the one suggested in \cite{RLZ19} as $\l(\bfx) \sim h$.

\section{A Numerical Illustration}\label{sec:S6}
This section presents a computational demonstration of the model's theoretical results, which are consistent with the theoretical predictions. 

\subsection*{$2D$ channel flow}
While  $2D$ turbulence is easier to simulate, the absence of a vortex stretching mechanism results in an inverse energy cascade \cite{AD06, FJMR02}, making the behavior of $2D$ turbulence more intricate compared to $3D$ turbulence. Although our analytical bound \eqref{FinalBd} is derived for \eqref{BLM} in three-dimensional space, we conduct a simulation in two-dimensional space due to the significant computational complexity and time consumption associated with three-dimensional simulations. 
%The corresponding analytical study of the backscatter Baldwin-Lomax model \eqref{BLM} in two-dimensional space will be deferred to future research.
\subsection*{Problem Setting} 

 We investigated the classical two-dimensional channel flow past an obstacle \cite{schfer1996benchmark}. The domain under consideration is a rectangular channel defined by \(\Omega = [0, 4] \times [0, 1]\), which includes a square obstacle located at \([0.5, 0.6] \times [0.45, 0.55]\), \Cref{Fig;2dDomain}. There is no external forcing (\(\bff = 0\)), and the flow passes through this domain from left to right. The inflow profile is specified as 
$$u_1|_{\mbox{inflow}} = 4y(1 - y)  \hspace{0.5cm}  \mbox{and}  \hspace{0.5cm} u_2|_{\mbox{inflow}} = 0,$$ 
while the no-slip condition \(\bfu = 0\) is imposed on the remaining boundaries. We implement the zero-traction boundary condition using the standard 'do-nothing' condition at the outflow.

We set $\beta = 10$ in all simulations, following the choice in \cite[p.~17]{RLZ19}.
 As suggested by our analytical results in \eqref{MeshIndp} and \eqref{MeshDep}, and to verify these findings numerically, we conduct two sets of simulations using different spatially dependent mixing lengths $l(\mathbf{x})$.  Let $\bar{y}$ denote the distance from a point $\mathbf{x}$ to the nearest wall. Motivated by \eqref{MeshIndp} and  the formulation in  \cite[Chapter 3]{W06}, we define the first mixing length $l_1(\mathbf{x})$ as
\begin{equation}\label{L1}
l_1(\mathbf{x}) =
\begin{cases}
0.41 \cdot \bar{y}, & \text{if } \bar{y} < 0.2 \cdot Re^{-1/2}, \\
0.41 \cdot 0.2 \cdot Re^{-1/2}, & \text{otherwise}.
\end{cases}
\end{equation}
Motivated by \eqref{MeshDep}, we also consider a second choice given by 
\begin{equation}\label{L2}
l_2(\mathbf{x}) = h^{2/3},
\end{equation}
where $h$ denotes the local mesh width.

\begin{figure}[h!]
    \centering
    \includegraphics[width=0.7\textwidth]{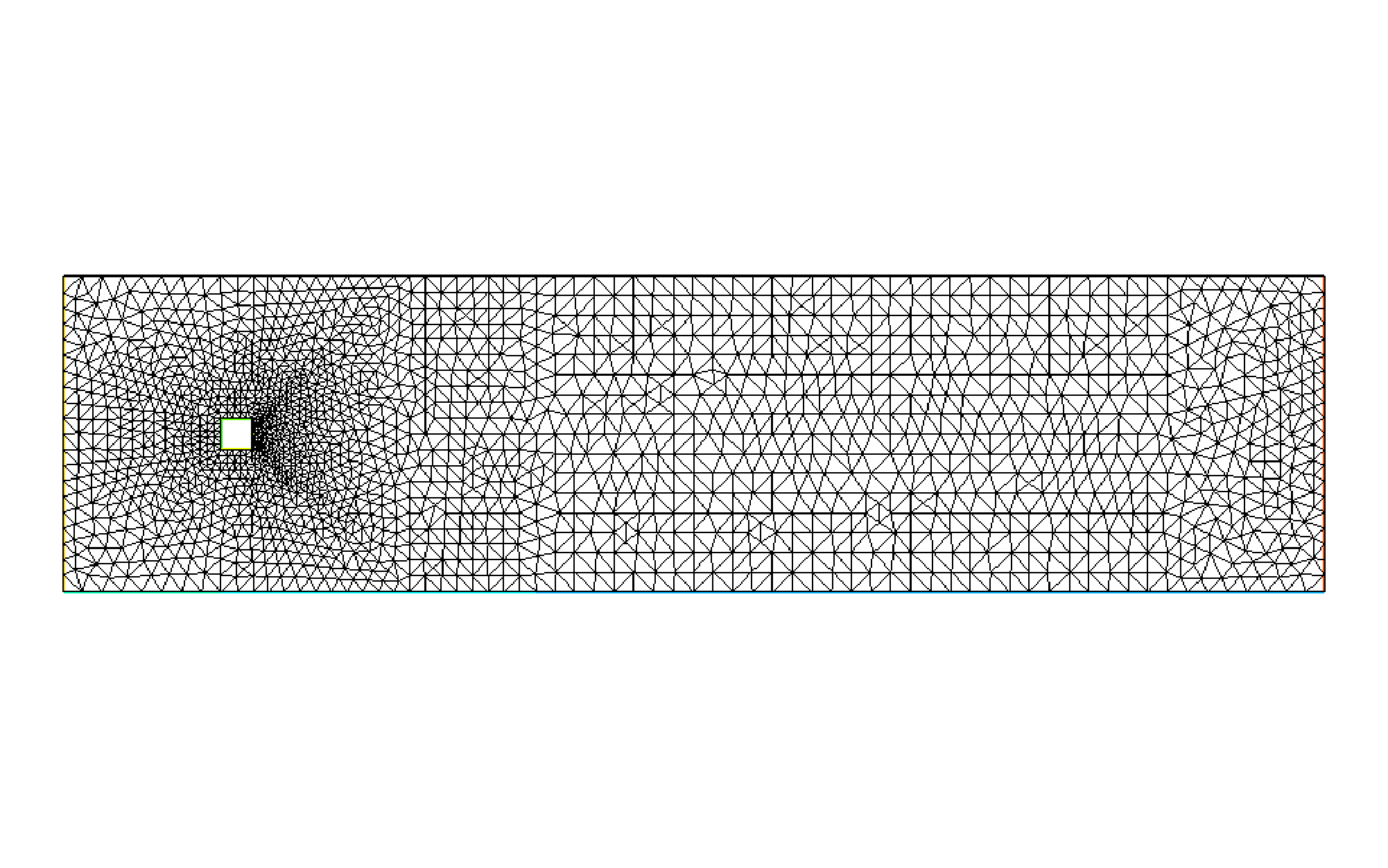} 
    \caption{Flow passing an obstacle; the unstructured mesh used in the numerical experiments.}
    \label{Fig;2dDomain}
\end{figure}
%\vspace{-0.2cm}

\begin{figure}[h!]
    \centering
    \includegraphics[width=1.0\textwidth]{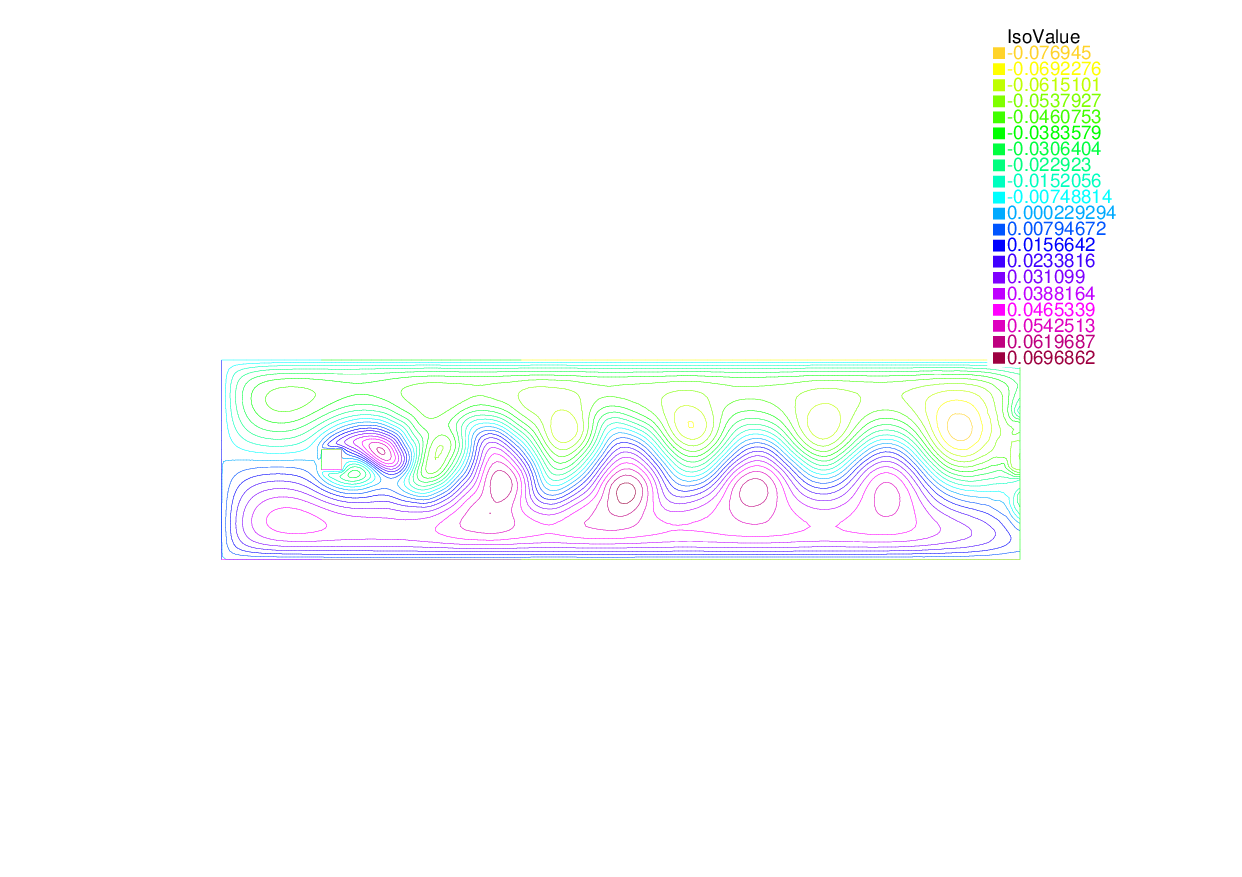} 
    \caption{Streamlines for $\Rey=500$ at $t=45$.}
    \label{Fig;StreamLine}
\end{figure}
We utilize the finite element method, specifically employing FreeFem++ \cite{FreeFEM} for our simulations. For spatial discretization, we implement P2-P1 Taylor-Hood mixed finite elements, while a second-order Crank-Nicolson scheme is applied for temporal discretization (see Chapter 9 of \cite{L08} for a comprehensive description of the scheme). The problem is computed on a Delaunay-Voronoi generated triangular mesh, which features a higher density of mesh points near the obstacle and a lower density in other areas, resulting in $17$K velocity degrees of freedom, as illustrated in \Cref{Fig;2dDomain}. We set $\Delta t = 0.01$ and conducted the simulation from $T = 0$ to $T = 50$. \Cref{Fig;StreamLine} shows a snapshot of the flow streamlines at $t = 45$ for $\Rey = 500$.

\subsection*{Energy Dissipation Rate} To numerically study the dependence of \(\eps\) on the Reynolds number, we conducted tests with varying Reynolds numbers five times, ranging from $100$ to $2000$, by adjusting the viscosity. In Figures~\ref{Fig;2dDRvsT} and~\ref{Fig;2dDRvsT2}, the instantaneous energy dissipation rate $\varepsilon(t)$ is plotted as a function of time for the two proposed mixing lengths, $l_1(\mathbf{x})$ and $l_2(\mathbf{x})$, respectively.   As expected and observed, the pointwise values appear disordered and unpredictable. Therefore, we instead focus on the statistical (time-averaged) quantity. After averaging the energy dissipation over time, we obtained five data points, each representing \(\eps\) corresponding to one of the Reynolds numbers. We employed Matlab's nonlinear least squares tool to fit the data to the form \(a + b \, \Rey^{-1}\). The initial guesses for the iterative optimization solver were set to \(a_0 = 0.5\) and \(b_0 = 5\).  Figures  \ref{Fig;2dEDR} and \ref{Fig;2dEDR2}  illustrate that the long-term average of the energy dissipation rates for the model scales as \(\eps \sim 1+ \Rey^{-1}\), consistent with our analysis given in \Cref{MainThm}, and the range of the mixing length suggested in \eqref{MeshIndp} and \eqref{MeshDep}.
\begin{figure}[h!]
    \centering
    \includegraphics[width=0.7\textwidth]{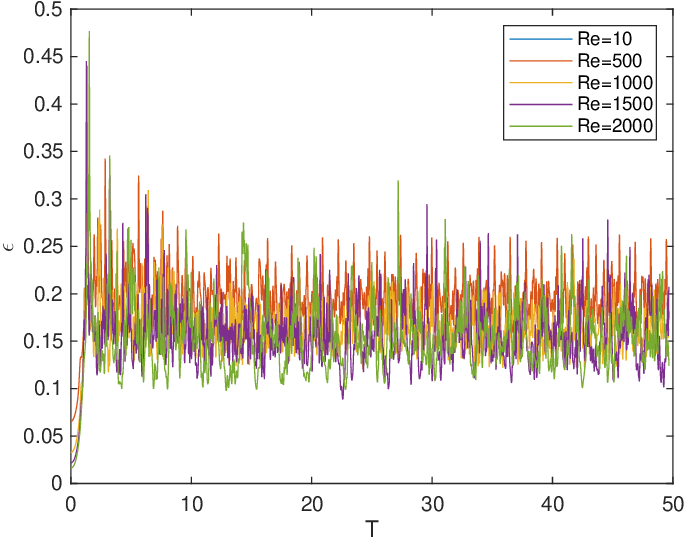} 
    \caption{Time evolution of the energy dissipation rate $\varepsilon$ using the mixing length $l_1(\mathbf{x})$ defined in~\eqref{L1}.}
    \label{Fig;2dDRvsT}
\end{figure}

\begin{figure}[h!]
    \centering
    \includegraphics[width=0.7\textwidth]{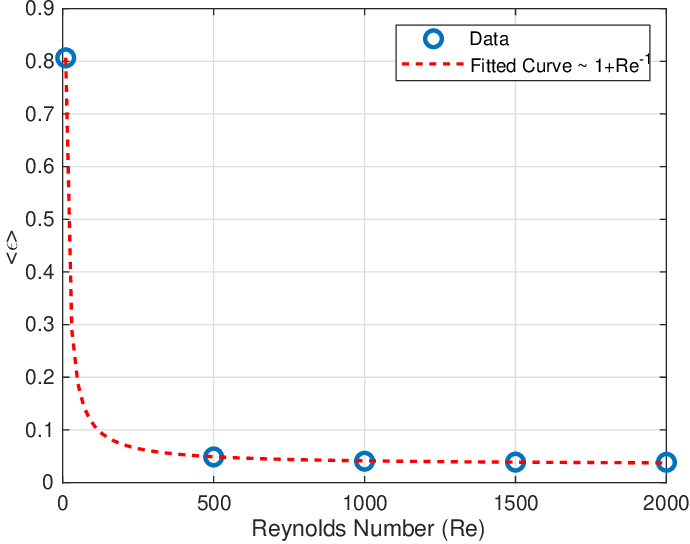} 
   \caption{The time-averaged energy dissipation rate $\varepsilon$ using the mixing length $l_1(\mathbf{x})$ defined in~\eqref{L1}.}
    \label{Fig;2dEDR}
\end{figure}

\begin{figure}[h!]
    \centering
    \includegraphics[width=0.7\textwidth]{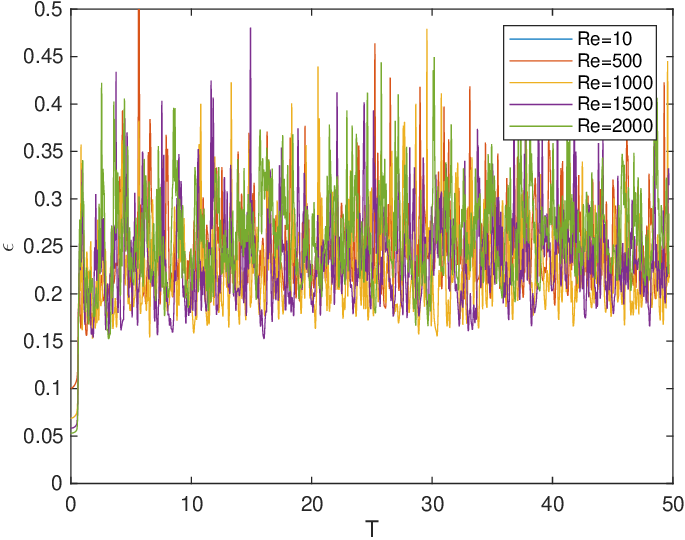} 
    \caption{Time evolution of the energy dissipation rate $\varepsilon$ using the mixing length $l_2(\mathbf{x})$ defined in~\eqref{L2}.}
    \label{Fig;2dDRvsT2}
\end{figure}

\begin{figure}[h!]
    \centering
    \includegraphics[width=0.7\textwidth]{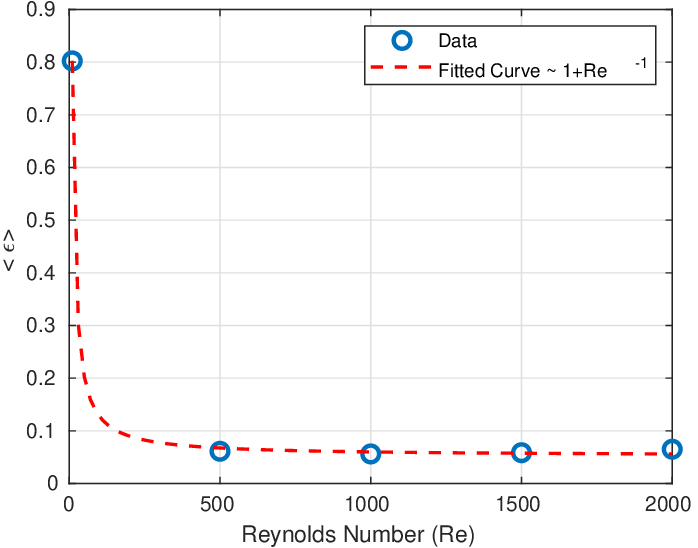} 
\caption{The time-averaged energy dissipation rate $\varepsilon$ using the mixing length $l_2(\mathbf{x})$ defined in~\eqref{L2}.}
\label{Fig;2dEDR2}
\end{figure}

%\subsection{3D Taylor-Couette flow}

\section{Discussion}\label{S7}

 This paper has addressed potential issues inherent in traditional eddy viscosity turbulence models, particularly concerning excessive dissipation and parameter specification. Considering the backscatter eddy viscosity model \eqref{BLM} suggested in \cite{RLZ19}, we rigorously derived an upper bound on the energy dissipation rate per unit mass. Theorem~\ref{MainThm} proves that the dissipation ratio for the backscatter eddy viscosity model~\eqref{BLM} remains uniformly bounded in the infinite Reynolds number limit, consistent with the heuristic scaling arguments outlined in the introduction. The analysis suggests an appropriate range for the model's mixing length, which is further supported by our numerical results.  In this manuscript, we also address the gap in well-posedness.

 Despite our progress, we recognize that calibrating certain model parameters remains challenging---especially the parameter \( \beta \), for which it is not yet clear whether it should be treated as constant or spatially variable. Unlike the mixing length $\l(\bfx)$, we have not identified a clear mathematical or physical justification to guide this choice.

%Furthermore, our findings indicate that this model exhibits statistical behavior similar to that of the simpler Baldwin-Lomax model, while providing compelling computational evidence of backscatter phenomena.

 The estimate on $\langle \varepsilon \rangle$ in \eqref{FinalBd} is expressed in terms of controlled quantities such as $\nu$, $F$, $L$, and a derived quantity $U$. However, realizing a specific value of $U$ remains elusive, leading to a lack of a priori information regarding the dissipation rate or the intricate structure of turbulent flow for a given viscosity and applied body force. Additionally,  extending our estimates of $\varepsilon$ to channel and shear flow cases remains an open problem, which would provide valuable insights into near-wall behavior. Furthermore, broadening our analysis to two-dimensional turbulence presents an intriguing challenge, necessitating bounds on not only energy but also enstrophy dissipation rate due to the energy cascade scenario in $2D$ \cite{FJMR02}.

\section*{Data Availability Statement}
The data that support the findings of this study are available from the corresponding author upon reasonable request.

\section*{Conflict of Interest}
The authors declare that there is no conflict of interest regarding the publication of this article.

\hfill

\noindent Ali Pakzad\\
{\footnotesize
Department of Mathematics\\
California State University Northridge \\
Email: \url{pakzad@csun.edu}
}

\hfill

\noindent Farjana Siddiqua\\
{\footnotesize
School of Mathematics\\
Georgia Institute of Technology \\
Email: \url{fsiddiqua3@gatech.edu}\\
}

\end{document}